\documentclass[reqno]{amsart}

\input xy
\xyoption{all}
\usepackage{epsfig}
\usepackage{color}
\usepackage{amsthm}
\usepackage{amssymb}
\usepackage{amsmath}
\usepackage{amscd}
\usepackage{amsopn}
\usepackage{url}
\usepackage{hyperref}\hypersetup{colorlinks}

\usepackage{color} 

\definecolor{darkred}{rgb}{1,0,0} 
\definecolor{darkgreen}{rgb}{0,0.8,0}
\definecolor{darkblue}{rgb}{0,0,1}

\hypersetup{colorlinks,
linkcolor=darkblue,
filecolor=darkgreen,
urlcolor=darkred,
citecolor=darkgreen}

\numberwithin{equation}{section}
\newtheorem {Theorem}{Theorem}
\numberwithin{Theorem}{section}

\newtheorem {Lemma}[Theorem]    {Lemma}

\newtheorem {Corollary}[Theorem]{Corollary}
\theoremstyle{definition}

\theoremstyle{remark}
\newtheorem{Remark}[Theorem]{Remark}

\def    \eps    {\epsilon}

\newcommand{\CA}{{\mathcal A}}

\newcommand{\CS}{{\mathcal S}}

\newcommand{\supp}{\operatorname{supp}}

\newcommand{\ff}{{\mathfrak f}}
\newcommand{\fc}{{\mathfrak c}}
\newcommand{\fh}{{\mathfrak h}}
\newcommand{\fF}{{\mathfrak F}}

\newcommand{\fy}{{\mathfrak y}}
\newcommand{\hook}{\hookrightarrow}

\newcommand{\tal}{\tilde{\alpha}}
\newcommand{\tx}{\tilde{x}}

\newcommand{\tpi}{\tilde{\pi}}

\newcommand{\A}{{\mathcal A}}

\def    \C      {{\mathbb C}}
\def    \R      {{\mathbb R}}

\def    \Z      {{\mathbb Z}}
\def    \N      {{\mathbb N}}
\def    \Q      {{\mathbb Q}}

\def    \T      {{\mathbb T}}
\def    \CP     {{\mathbb C}{\mathbb P}}

\def    \12    {{\frac{1}{2}}}

\def    \p      {\partial}

\def    \Ext  {\operatorname{Ext}}
\def    \Hom  {\operatorname{Hom}}

\def    \HF     {\operatorname{HF}}
\def    \HC     {\operatorname{HC}}
\def    \CC     {\operatorname{CC}}

\def    \H     {\operatorname{H}}

\def    \PSL    {\operatorname{PSL}}

\def    \Tors    {\operatorname{Tors}}

\def    \MUCZ  {\operatorname{\mu_{\scriptscriptstyle{CZ}}}}

\begin{document}


\setlength{\smallskipamount}{6pt}
\setlength{\medskipamount}{10pt}
\setlength{\bigskipamount}{16pt}





\title[Conley Conjecture for Reeb Flows]{On the Conley Conjecture for
  Reeb Flows}

\author[Viktor Ginzburg]{Viktor L. Ginzburg}
\author[Ba\c sak G\"urel]{Ba\c sak Z. G\"urel}
\author[Leonardo Macarini]{Leonardo Macarini}

\address{Department of Mathematics, UC Santa Cruz, Santa Cruz, CA
  95064, USA} \email{ginzburg@ucsc.edu} 

\address{Department of Mathematics, University of Central Florida,
  Orlando, FL 32816, USA} \email{basak.gurel@ucf.edu}

\address{Universidade Federal do Rio de Janeiro, Instituto de
  Matem\'atica, Cidade Universit\'aria, CEP 21941-909, Rio de Janeiro,
  Brazil} \email{leonardo@impa.br}


\date{\today} 

\thanks{The work is partially supported by NSF grants DMS-1308501 (VG)
  and DMS-1414685 (BG) and the CNPq, Brazil (LM)}

\bigskip

\begin{abstract}   
In this paper we prove the existence of infinitely many closed Reeb
orbits for a certain class of contact manifolds. This result can be
viewed as a contact analogue of the Hamiltonian Conley conjecture. The
manifolds for which the contact Conley conjecture is established are
the pre-quantization circle bundles with aspherical base. As an
application, we prove that for a surface of genus at least two with a
non-vanishing magnetic field, the twisted geodesic flow has infinitely
many periodic orbits on every low energy level.
\end{abstract}

\maketitle

\tableofcontents


\section{Introduction}
\label{sec:intro}
In this paper we establish a contact analogue of the Hamiltonian
Conley conjecture -- the existence of infinitely many periodic orbits
-- for Reeb flows on the pre-quantization circle bundles with
aspherical base. As an application, we prove that for a surface of
genus at least two with non-vanishing magnetic field, the twisted
geodesic flow has infinitely many periodic orbits on every low energy
level.

To put these results in perspective, recall that the Hamiltonian
Conley conjecture asserts the existence of infinitely many periodic
orbits for every Hamiltonian diffeomorphism of a closed symplectic
manifold whenever the manifold meets some natural general
requirements. This is the case for manifolds with
spherically-vanishing first Chern class (of the tangent bundle) and
for negative monotone manifolds; see \cite{CGG, GG:gaps, He} and also
\cite{FH, Gi:CC, GG:nm, Hi, LeC, Ma,SZ}. It is important to note,
however, that the Conley conjecture, as stated, fails for some simple
manifolds such as $S^2$: an irrational rotation of the sphere about
the $z$-axis has only two periodic orbits, which are also the fixed
points; these are the poles. In fact, any manifold that admits a
Hamiltonian torus action with isolated fixed points also admits a
Hamiltonian diffeomorphism with finitely many periodic orbits. Among
these manifolds are $\CP^n$, the Grassmannians, and, more generally,
most of the coadjoint orbits of compact Lie groups as well as
symplectic toric manifolds. (There is also a variant of the Conley
conjecture for such manifolds, considered in \cite{Gu:hyp,Gu:nc} and
inspired by a celebrated theorem of Franks, \cite{Fr1,Fr2}, but this
conjecture is not directly related to our discussion.)

To summarize, the collection of all closed symplectic manifolds
naturally breaks down into two classes: those for which the Conley
conjecture holds and those for which the Conley conjecture fails. The
non-trivial assertion is then that the former class is non-empty and
even quite large. At this stage, we are far from understanding where
exactly the dividing line between the two classes is, but at least on
the level of proofs there seems to be no connection between the Conley
conjecture and the additive structure of the (ordinary or quantum)
homology of the manifold.

The situation with closed contact manifolds looks more involved even
if we leave aside such fundamental questions as the Weinstein
conjecture.

First of all, there is a class of contact manifolds for which every
Reeb flow has infinitely many closed orbits because the rank of
contact or symplectic homology grows as a function of index or some
other parameter connected with the order of iteration. This
phenomenon, studied in \cite{HM,McL}, generalizes and is inspired by
the results of \cite{GM} establishing the existence of infinitely many
closed geodesics for manifolds such that the homology of the free loop
space grows. (A technical but important fact underpinning the proof is
that the iterates of a given orbit can make only bounded contributions
to the homology; see \cite{GG:gap,GM,HM,McL} for various incarnations
of this result.)  By \cite{VPS} and \cite{AS,SW,Vi:f}, among contact
manifolds in this class are the unit cotangent bundles $ST^*M$
whenever $\pi_1(M)=0$ and the algebra $\H^*(M;\Q)$ is not generated by
one element, and some others; \cite{HM,McL}. This homologically forced
existence of infinitely many Reeb orbits has very different nature
from the Hamiltonian Conley conjecture where there is no growth of
homology: the Floer homology of a Hamiltonian diffeomorphism of a
closed manifold obviously does not change with iterations.

Then there are contact manifolds admitting Reeb flows with finitely
many closed orbits. Among these are, of course, the standard contact
spheres and, more generally, pre-quantization circle bundles over
symplectic manifolds admitting torus actions with isolated fixed
points (see \cite[Example 1.13]{Gu:pr}) including the Katok--Ziller
flows. Another important group of examples, also containing the
standard spheres, arises from contact toric manifolds; see
\cite{AM}. Note that these two classes overlap, but do not entirely
coincide.

Finally, there is, as we show in this paper, a \emph{non-empty} class
of contact manifolds for which every Reeb flow (meeting certain
natural index conditions) has infinitely many closed orbits, although
there is no obvious homological growth -- the rank of the relevant
contact homology remains bounded. One can expect the class of
manifolds for which the conjecture holds to be quite large, but at
this point we can prove such unconditional existence of infinitely
many closed Reeb orbits only for pre-quantization circle bundles of
aspherical manifolds. (Moreover, for the sake of simplicity, we also
make an additional assumption that the first Chern class of the
contact structure is atoroidal. Note also that the index conditions
mentioned above play a purely technical role, but are inherent in the
construction of the cylindrical contact homology utilized in the
proof.) This variant of the contact Conley conjecture is one of the
main results of the paper (Theorem \ref{thm:main}), and its proof,
drawing from \cite{GG:gaps} and also \cite{GHHM,HM}, clearly shows the
similarity of the phenomenon with the Hamiltonian Conley conjecture.

This picture is, of course, oversimplified and certainly not even
close to covering all the range of possibilities, even on the
homological level. (For instance, hypothetically, Reeb flows for
overtwisted contact structures have infinitely many periodic orbits,
but where should one place such contact structures in our
``classification''? See \cite{El,Yau} and also \cite{BvK} for further
details.)

It is also worth pointing out that our proof of the contact Conley
conjecture heavily relies, in some instances beyond the formal level,
on the machinery of cylindrical contact homology (see, e.g.,
\cite{Bo,Bo:thesis,SFT} and references therein), which is yet to be
fully put on a rigorous basis (see \cite{HWZ:SC,HWZ:poly}).

As an application of our main result, we prove the existence of
infinitely many periodic orbits for all low energy levels of twisted
geodesic flows on surfaces with non-vanishing magnetic field (Theorem
\ref{thm:magnetic}).

To be more precise, consider a closed Riemannian manifold $M$ and let
$\sigma$ be a closed 2-form on $M$. Equip $T^*M$ with the twisted
symplectic structure $\omega=\omega_0+\pi^*\sigma$, where $\omega_0$
is the standard symplectic form on $T^*M$ and $\pi\colon T^*M\to M$ is
the natural projection, and let $K$ be the standard kinetic energy
Hamiltonian on $T^*M$ corresponding to a Riemannian metric on $M$.
The Hamiltonian flow of $K$ on $T^*M$ describes the motion of a charge
on $M$ in the magnetic field $\sigma$ and is referred to as a twisted
geodesic or magnetic flow. In contrast with the geodesic flow (the
case $\sigma=0$), the dynamics of the twisted geodesic flow on an
energy level depends on the level. In particular, when $M$ is a
surface of genus $g \geq 2$, the example of the horocycle flow shows
that a symplectic magnetic flow need not have periodic orbits on all
energy levels. Note also that the dynamics of a twisted geodesic flow
crucially depends on whether one considers low or high energy levels
and, at least from a technical perspective, on whether $\sigma$ is
assumed to be exact or symplectic.

The existence problem for periodic orbits of a charge in a magnetic
field was first addressed in the context of symplectic geometry by
V.I. Arnold in the early 80s; see \cite{Ar:fs,Ar88}. Namely, Arnold
proved, as a consequence of the Conley--Zehnder theorem, the existence
of periodic orbits of a twisted geodesic flow on $\T^2$ with
symplectic magnetic field for all energy levels when the metric is
flat and all low energy levels for an arbitrary metric,
\cite{Ar88}. (It is still unknown if the second of these results can
be extended to all energy levels.)  Since Arnold's work, the problem
has been studied in a variety of settings. We refer the reader to,
e.g., \cite{Gi:newton} for more details and further references prior
to 1996 (see also \cite{Ta}) and to, e.g.,
\cite{AMP,AMMP,CMP,GG:capacity,GG:wm,Ke:m,Sc,Scn2,Us} for by far an
incomplete list of more recent results.

Here we focus on the case where the magnetic field form $\sigma$ is
symplectic (e.g., non-vanishing when $\dim M=2$), and we are
interested in dynamics on low energy levels. In this setting, in all
dimensions, the existence of at least one closed orbit on every
sufficiently low energy level was proved in \cite{GG:wm,Us}. It was
also conjectured, and proved for $M=\T^2$, in \cite{GG:wm} that in
fact every low energy level carries infinitely many periodic orbits
when $M$ is symplectically aspherical. (Although this conjecture is
merely one in a sequence of such hypothetical lower bounds (see, e.g.,
\cite{Ar:fs,Ar88,Gi:FA,Gi:mathz,Ke:m}), it differs from the previous
ones in that it takes into account periodic orbits of arbitrarily
large period, but not just the ``short'' orbits.) Thus the result of
the present paper completes the proof of the conjecture in the case
where $M$ is a surface. (See Remark~\ref{rmk:conj} for a further
discussion.)

\subsection*{Acknowledgments}
The authors are grateful to Fr\'ed\'eric Bourgeois for useful
discussions and remarks. A part of this work was carried out while
the first two authors were visiting ICMAT, Madrid, Spain, and IMBM,
\.{I}stanbul, Turkey, and also during the first author's visit to the
National Center for Theoretical Sciences (South), Taiwan and the third
author's visit to the Chern Institute of Mathematics, Tianjin, China.
The authors would like to thank those institutes for their warm
hospitality and support.

\section{Main results}
\label{sec:results}

\subsection{Conley conjecture for pre-quantization circle bundles}
\label{sec:main}
Consider a closed symplectic manifold $(M^{2n},\omega)$ such that the
form $\omega$, or more precisely its cohomology class $[\omega]$, is
\emph{integral}, i.e., $[\omega]\in \H^2(M;\Z)/\Tors$. Let $\pi\colon
P\to M$ be an $S^1$-bundle over $M$ with the first Chern class
$-[\omega]$. The bundle $P$ admits an $S^1$-invariant 1-form
$\alpha_0$ such that $d\alpha_0=\pi^*\omega$ and $\alpha_0(R_0)=1$,
where $R_0$ is the vector field generating the $S^1$-action on $P$. In
other words, when we set $S^1=\R/\Z$ and identify the Lie algebra of
$S^1$ with $\R$, the form $\alpha_0$ is a connection form on $P$ with
curvature $\omega$; see, e.g., \cite[Appendix A]{GGK} for a detailed
discussion of various sign and other conventions used in this setting.

Clearly, $\alpha_0$ is a contact form with Reeb vector field $R_0$,
and the connection distribution $\xi=\ker\alpha_0$ is a contact
structure on $P$. Up to a gauge transformation, $\xi$ is independent
of the choice of $\alpha_0$. The circle bundle $P$ equipped with this
contact structure or contact form is usually referred to as a
\emph{pre-quantization circle bundle} or a Boothby--Wang bundle.  Also
recall that a degree two (real) cohomology class on $P$ is said to be
\emph{atoroidal} if its integral over any smooth map $\T^2\to P$ is
zero. (Note that such a class is necessarily aspherical.)  Finally, in
what follows we will denote by $\ff$ the free homotopy class of the
fiber of $\pi$.

The main tool used in this paper is the cylindrical contact
homology. As is well known, to have this homology defined for a
contact form $\alpha$ on any closed contact manifold $P$ one has to
impose certain additional requirements on closed Reeb orbits of
$\alpha$; \cite{Bo,SFT}. Namely, we say that a non-degenerate contact
form $\alpha$ is \emph{index--admissible} if its Reeb flow has no
contractible closed orbits with Conley--Zehnder index $2-n$ or $2-n\pm
1$. (This not the standard term. Note also that $\dim P=2n+1$.) In
general, $\alpha$ or its Reeb flow is index--admissible when there
exists a sequence of non-degenerate index--admissible forms
$C^1$-converging to $\alpha$.

This requirement is usually satisfied once $(P,\alpha)$ has some
geometrical convexity properties. For instance, the Reeb flow on a
strictly convex hypersurface in $\R^{2m}$ is index--admissible;
\cite{HWZ:convex}. Likewise, as is observed in \cite{Be}, the twisted
geodesic flow on a low energy level for a symplectic magnetic field on
a surface of genus $g\geq 2$ is index admissible; see the proof of
Theorem \ref{thm:magnetic} for more details.

Recall also that closed orbit is said to be \emph{weakly
  non-degenerate} when at least one of its Floquet multipliers is
different from 1; cf.\ \cite{SZ}. A form $\alpha$ or its Reeb flow is
weakly non-degenerate when all closed Reeb orbits are weakly
non-degenerate. For instance, a non-degenerate orbit (flow) is weakly
non-degenerate. Clearly, weak non-degeneracy is a $C^\infty$-generic
condition.

The main result of the paper is the following theorem, which, as is
pointed out in the introduction, can be viewed as an analogue of the
Conley conjecture for contact forms $\alpha$ on the pre-quantization
bundle $P$ over $M$, supporting the (cooriented) contact structure
$\xi$, i.e., such that $\ker\alpha=\xi$ and $\alpha(R_0)>0$.

\begin{Theorem}[Contact Conley Conjecture] 
\label{thm:main}
Assume that 
\begin{itemize}
\item[(i)]  $M$ is aspherical, i.e., $\pi_r(M)=0$ for all $r\geq 2$, and
\item[(ii)] $c_1(\xi)\in \H^2(P;\R)$ is atoroidal.
\end{itemize}
Let $\alpha$ be an index--admissible contact form on $P$ supporting
$\xi$. Then the Reeb flow of $\alpha$ has infinitely many closed
orbits with contractible projections to $M$. Assume furthermore that
the Reeb flow has finitely many closed Reeb orbits in the free
homotopy class $\ff$ of the fiber and that these orbits are weakly
non-degenerate. Then for every sufficiently large prime $k$ the Reeb
flow of $\alpha$ has a simple closed orbit in the class $\ff^k$.
\end{Theorem}

Note that under the conditions of the theorem, all iterates $\ff^k$,
$k\in\N$, are distinct; see the discussion below and Lemma
\ref{lemma:fiber}. Hence, the second part of the theorem implies the
existence of infinitely many periodic orbits and is really a
refinement of the first part, under the weak non-degeneracy condition.

\begin{Remark}[Growth] 
\label{rmk:growth}
It readily follows from Theorem \ref{thm:main} that, when the Reeb
flow of $\alpha$ is weakly non-degenerate, the number of simple
periodic orbits of the Reeb flow of $\alpha$ with period (or
equivalently action) less than $a\gg 0$ is bounded from below by $C
\cdot a/\ln a$, where $C >0$ depends only on $\alpha$.  In fact, we
have the lower bound $C_0 \cdot a/\ln a - C_1$, where $C_0=\inf
\alpha(R_0)$ and $C_1$ depends only on $\alpha$. These growth lower
bounds are typical for the Hamiltonian Conley conjecture type results
mentioned in the introduction; see also \cite{Hi:geod} for the case of
closed geodesics on $S^2$. (In dimension two, however, stronger growth
results have been established in some cases; see, e.g., \cite{LeC} and
\cite[Prop.\ 4.13]{Vi} and also \cite{BH,FH,Ke}.)  Note also that the
weak non-degeneracy requirement here plays a technical role and
probably can be eliminated; see Remark \ref{rmk:difficulty}.
\end{Remark}

Finally, we would like to point out a similarity between Theorem
\ref{thm:main} and the main result of \cite{Gu:nc} where a variant of
the Hamiltonian Conley conjecture is established for non-contractible
orbits.

\subsection{Discussion: topological conditions}
\label{sec:discussion}
The conditions on $(P,\xi)$ imposed in the theorem and, in particular,
condition (i), i.e., the requirement that $M$ is an Eilenberg--MacLane
space $K(\pi,1)$, deserve a detailed discussion.

First, however, let us introduce some notation to be used throughout
the paper. Let $\tpi_1(P)$, where $P$ is an arbitrary manifold, be the
collection of the free homotopy classes of maps $S^1\to P$. We
identify $\tpi_1(P)$ with the set of conjugacy classes in
$\pi_1(P)$. Although in general $\tpi_1(P)$ is not a group, the powers
of an element of this set are well defined. We denote by 1, the image
of the unit in $\tpi_1(P)$.

The role of condition (i) in the proof of the theorem is
two-fold. Namely, it can be replaced by the following two conditions,
which are both consequences of~(i):

\begin{itemize}

\item[(i-a)] the class $[\omega]$ is aspherical, i.e.,
  $\omega|_{\pi_2(M)}=0$,
\item[(i-b)] the fundamental group $\pi_1(M)$ has no torsion.

\end{itemize}
A possibly non-obvious point here is that (i) implies (i-b). (This
fact, for any finite-dimensional CW-complex $M$, is sometimes
attributed to P.A. Smith. Here is a proof taken from \cite{Lu}: Recall
that $\H^*(\Z_k;\Z)=\Z_k$ for all even degrees $*$. Let $\Z_k\subset
\pi_1(M)$ be a finite cyclic subgroup. We can take $\tilde{M}/\Z_k$,
where $\tilde{M}$ is the universal covering of $M$, as the classifying
space $B\Z_k$. If $M$ were finite-dimensional, we would have
$\H^*(\Z_k;\Z)=\H^*(B\Z_k;\Z)=0$ for $*>\dim M$. A contradiction.)

We will show in Section \ref{sec:fiber} that when (i-a) holds, all
free homotopy classes $\ff^k$, where $k\in\N$ and $\ff\in\tpi_1(P)$ is
the class of a fiber, are distinct and none of these classes is
trivial. This fact is used in a variety of ways, e.g., to ensure that
the natural grading of the contact homology by the free homotopy
classes $\ff^k$, $k\in\N$, is actually a grading by $\N$. This is
essential because the class $\ff^k$ takes the role of the order of
iteration $k$ in the proof of the Hamiltonian Conley conjecture.

Condition (i-b) is used in the proof to show that the free homotopy
class $\ff$ is primitive and that, more generally, for every $k\in\N$
the only solutions $\fh\in\tpi_1(P)$ and $l\geq 0$ of the equation
$\fh^l=\ff^k$ are $\fh=\ff^r$, for some $r\in\N$, and $l=k/r$; see
Lemma \ref{lemma:fiber2}. These properties of $\ff$ are needed to
guarantee that the closed Reeb orbits detected by the filtered contact
homology of $\alpha$ are simple. Condition (i-b) plays an absolutely
crucial, albeit technical, role in the proof of Theorem \ref{thm:sdm}.

On the other hand, condition (i-a) is clearly necessary for the
theorem.  This condition obviously fails for, say, $\CP^n$ or complex
Grassmannians, as does the assertion of the theorem. In fact, both
condition (i-a) and the assertion of the theorem fail whenever the
base $M$ admits a Hamiltonian circle action with isolated fixed
points; cf.\ \cite[Example 1.14]{Gu:pr}. As is shown in that example,
this fact is essentially the reason for the existence of asymmetric
Finsler metrics with finitely many closed geodesics on, say, the
spheres; see \cite{Ka} and also \cite{Zi}.

Finally, condition (ii) is imposed only for the sake of simplicity and
can be eliminated once a suitably defined Novikov ring is incorporated
into the contact homology.  This condition is automatically met when
$(M,\omega)$ is monotone or negative monotone, i.e., $c_1(TM)=\lambda
[\omega]$ in $\H^2(M;\R)$ (but not only on $\pi_2(M)$) for some
$\lambda\in\R$.

All conditions of the theorem on $(P,\xi)$ are obviously satisfied
when $M$ is a surface of genus greater than or equal to one or when
$M=\T^{2n}$ or for negative monotone (or monotone, if they exist)
hyperbolic K\"ahler manifolds. Furthermore, the requirements of the
theorem are met by the product $M_1\times M_2$ whenever they are met
by $M_1$ and~$M_2$.

\begin{Remark}
  It is worth pointing out that the $S^1$-bundle $\pi\colon P\to M$ is
  not unique and is not quite determined by $[\omega]$. Consider the
  ``duality'' exact sequence
$$
0\to\Ext_\Z(\H_1(M;\Z);\Z)\to \H^2(M;\Z)\to\Hom_{\Z}(\H_2(M;\Z);\Z)\to 0.
$$
We can identify the last term in this sequence with the integral de
Rham cohomology $\H^2(M;\Z)/\Tors$, which the class $[\omega]$ belongs
to, and the first term with $T=\Tors(\H_1(M;\Z))$. Thus we have
$$
0\to T\to \H^2(M;\Z)\to \H^2(M;\Z)/ \Tors\to 0 .
$$
The $S^1$-bundle $\pi\colon P\to M$ is uniquely determined by its
first Chern class $u$ which is a lift of $-[\omega]$ to $\H^2(M;\Z)$,
but not just by $-[\omega]$. (Of course, there is no ambiguity when
$T=0$.)  Finally, it is not hard to see from the proof of Theorem
\ref{thm:main} that condition (i-b) can be replaced by the requirement
that for any cyclic subgroup $G\subset \pi_1(M)$, the pull-back of $u$
to $\H^2(G;\Z)$ is zero.
\end{Remark}

\subsection{Application: a charge in a magnetic field}
\label{sec:magnetic}
Let now $M$ be a closed orientable surface equipped with a Riemannian
metric and $\sigma$ be a closed two-form (a magnetic field) on
$M$. The two-form $\omega=\omega_0+\pi^*\sigma$, where $\pi$ is the
natural projection $T^*M\to M$ and $\omega_0$ is the standard
symplectic form on $T^*M$, is symplectic. Here we will asume that
$\sigma$ is also symplectic (i.e., non-vanishing, since $M$ is a
surface). As is mentioned in the introduction, the motion of a unit
charge on $M$ is governed by the twisted geodesic flow, i.e., the
Hamiltonian flow with respect to $\omega$ of the standard kinetic
energy Hamiltonian $K\colon T^*M\to \R$, given by $K(p)=\|p\|^2/2$ in
self-explanatory notation.

Let us now focus on the levels $P_\eps=\{K=\eps\}$ for small values of
$\eps>0$.

\begin{Theorem}
\label{thm:magnetic}
Assume that $M$ has genus $g\geq 2$. Then for every small $\eps>0$,
the flow of $K$ has infinitely many simple periodic orbits on $P_\eps$
with contractible projections to $M$. Moreover, assume that the flow
has finitely many periodic orbits in the free homotopy class $\ff$ of
the fiber. Then, without any non-degeneracy assumptions, for every
sufficiently large prime $k$ there is a simple periodic orbit in the
class $\ff^k$.
\end{Theorem} 

We will prove this theorem in Section \ref{sec:pf-magn}.

\begin{Remark}
\label{rmk:conj}
As was observed in \cite[Prop.\ 1.5]{GG:wm}, an analogous result also
holds when $M=\T^2$. This is an immediate application of the Conley
conjecture proved in this case in \cite{FH}. Furthermore, it was
conjectured in \cite{GG:wm} that in all dimensions the twisted
geodesic flow has infinitely many periodic orbits on every low energy
level whenever $\sigma$ is symplectic and $(M,\sigma)$ is
symplectically aspherical.  Thus Theorem \ref{thm:magnetic} settles
the two-dimensional case of this conjecture. We see no reason why an
analogue of Theorem \ref{thm:magnetic} should hold when
$M=S^2$. Although this is not entirely obvious, we tend to think that
an example can be found by applying a variant of the Katok--Ziller
construction, \cite{Ka,Zi}, to a constant magnetic field on $S^2$ to
obtain a twisted geodesic flow with symplectic $\sigma$ and finitely
many simple orbits on (some) arbitrarily low energy levels; cf.\
\cite[Theorem 1.3]{Scn1} and \cite[Section 7]{GG:capacity}.
\end{Remark}

\begin{Remark} When, in the setting of Theorem \ref{thm:magnetic}, a
  twisted geodesic flow has finitely many periodic orbits in the free
  homotopy class of the fiber, the growth lower bounds from Remark
  \ref{rmk:growth} apply without any non-degeneracy
  assumptions. Furthermore, we can replace the contact action $a$ by
  the period of the Hamiltonian flow, but not in general by the
  Hamiltonian action. Finally, note that when $M$ is a surface of
  genus $g\geq 1$, condition (i-b) is easy to verify
  geometrically. Indeed, for $M=\T^2$, condition (i-b) obviously
  holds.  For $g\geq 2$, $\pi_1(M)$ is a subgroup of the group of
  isometries $\PSL(2;\R)$ of the hyperbolic plane. As is well known,
  the elements of $\pi_1(M)$ are necessarily hyperbolic isometries (or
  hyperbolic or parabolic, when $M$ is not compact), and hence have
  infinite order.
\end{Remark}

\section{Preliminaries: contact homology}
\label{sec:ch}
Our goal in this section is to review the definitions and results
concerning contact homology necessary for the proof and to set our
conventions.

\subsection{Generalities: cylindrical and linearized contact homology}
\label{sec:cch}

We start our discussion by briefly recalling the definition and basic
properties of the cylindrical and linearized contact homology in the
setting we are interested in. Our goal here is to mainly set our
conventions and notation. We refer the reader to, e.g., \cite{Bo,SFT}
and references therein, for a much more detailed account.

Let $(P^{2n+1},\xi)$ be a closed contact manifold with atoroidal first
Chern class $c_1(\xi)$. Fix a free homotopy class $\ff\in\tpi_1(P)$.

Let $\alpha$ be a non-degenerate, index--admissible (i.e., without
contractible closed Reeb orbits of index $2-n$ or $2-n\pm 1$), contact
form supporting $\xi$. The cylindrical contact homology
$\HC_*(\xi;\ff)$ of $\xi$ for the class $\ff$ is the homology of a
certain complex $\CC_*(\alpha;\ff)$ generated by the (good) closed
Reeb orbits of $\alpha$ in the homotopy class $\ff$ with the
differential counting rigid holomorphic cylinders in the
symplectization of $P$ asymptotic to periodic orbits. Although the
complex obviously depends on $\alpha$ (and some auxiliary structures),
its homology is well-defined and, in particular, independent of the
form.

Likewise, for $a<b$ outside the action spectrum $\CS(\alpha)$, the
filtered complex $\CC^{(a,\,b)}_*(\alpha;\ff)$ is generated by the
orbits $x$ with action
\begin{equation}
\label{eq:action}
\CA(x):=\int_x\alpha
\end{equation}
in the interval $I:=(a,\,b)$. The resulting homology
$\HC^{I}_*(\alpha;\ff)$ depends on $\alpha$, but not on the auxiliary
structures. Furthermore, it is invariant under deformations of the end
points $a$ and $b$ and the form $\alpha$ as long as the end points
remain outside the action spectrum; see \cite[Proposition 5]{GHHM}.
In particular, the homology is also defined ``by continuity'' for
degenerate index--admissible contact forms. (Recall from Section
\ref{sec:main} that a degenerate form is index--admissible if there
exists a sequence of non-degenerate index--admissible forms
$C^1$-converging to $\alpha$.)  When $a=-\infty$ and $b=+\infty$, we
recover the total cylindrical contact homology $\HC_*(\xi;\ff)$.

We note a minor difference of this definition from the standard one,
where $\CC^{I}_*(\alpha;\ff)$, for $I=(a,\,b)$, is set to be the
quotient
$\CC^{(0,\,b)}_*(\alpha;\ff)/\CC^{(0,\,a)}_*(\alpha;\ff)$. The
advantage of this approach is that, similarly to the Hamiltonian case
(see, e.g., \cite[Section 4.2.1]{GG:capacity}) the complex
$\CC^{I}_*(\alpha;\ff)$ is defined and, as is easy to see, its
homology is equal the standard $\HC^{I}_*(\alpha;\ff)$ even when only
the closed Reeb orbits with action in the window $I$ are
non-degenerate and, when contractible, have index different from $2-n$
or $2-n\pm 1$, provided that the regularity requirements are met.

The grading of the cylindrical contact complex and the homology
deserves a special discussion. First, note that to have the
Conley--Zehnder index $\MUCZ(x)$ of a closed non-degenerate Reeb orbit
$x$ in the class $\ff$ defined, we need to have a trivialization of
$\xi |_x$. The standard recipe calls for fixing a trivialization (up
to homotopy) of $\xi$ along a reference loop in $\ff$. Connecting $x$
to the reference loop by a cylinder and extending the trivialization
along the cylinder, we obtain a well-defined (up to homotopy)
trivialization of $\xi$ along $x$. Then the condition that $c_1(\xi)$
is atoroidal guarantees that the resulting trivialization is
independent of the cylinder. In what follows we will work with the
collection of classes $\ff^k$, $k\in\N$, assuming, unless $\ff=1$,
that all classes $\ff^k$ are distinct, and none of these classes is
trivial.  For our purposes it is crucial to choose trivializations
compatible with iterations. In other words, we fix a trivialization of
$\xi$ along a loop in the class $\ff$ and the trivialization for the
class $\ff^k$ is then obtained by taking the $k$-th iteration, in the
obvious sense, of this trivialization. (It is essential at this point
that all classes $\ff^k$ are distinct and none of them is trivial.)

Then we also have the mean index $\Delta(x)$ well-defined regardless
of whether $x$ is degenerate or not. (See, e.g., \cite{Lo,SZ} for the
definition of the mean index and its properties.) Moreover, our
convention guarantees that the mean index is homogeneous:
\begin{equation}
\label{eq:hom}
\Delta(x^r)=r\Delta(x)
\end{equation}
for all $r\in\N$, where $x$ is, of course, in one of the free homotopy
classes $\ff^k$. Furthermore, recall that for any choice of
trivializations we have
\begin{equation}
\label{eq:indexes}
|\MUCZ(\tx)-\Delta(x)|\leq n
\end{equation}
for every sufficiently small non-degenerate perturbation $\tx$ of $x$,
and that the inequality is strict when $x$ is weakly non-degenerate.

When $\ff=1$, the condition that $c_1(\xi)$ is atoroidal can be
relaxed, and it suffices to require this class to be aspherical:
$c_1(\xi)|_{\pi_2(P)}=0$. Then, as is well known, every contractible
closed Reeb orbit carries a canonical (up to homotopy) trivialization
and \eqref{eq:hom} holds automatically.

Finally, in the context of this paper it is much more convenient to
depart from the standard convention and have the contact homology
graded by the Conley--Zehnder index of the orbit without the shift of
degree by $(n+1)-3$. (Note that the dimension of $P$ is $2n+1$.) Thus,
throughout the paper, \emph{the contact homology is graded by the
  Conley--Zehnder index}.

\begin{Remark}
  If instead of assuming that $c_1(\xi)$ is atoroidal we imposed a
  stronger condition that $c_1(\xi)=0$ in $\H^2(P;\Z)$, we could have
  obtained a trivialization of $\xi$ along every loop, compatible with
  iterations, by fixing a non-vanishing section, up to homotopy, of
  the determinant bundle $\wedge_\C^n\xi$; cf.\ \cite{Es,GGo}. We also
  note that, when $\alpha$ is non-degenerate, our definition of the
  filtered contact homology still makes sense even if the end-points
  of $I$ are in $\CS(\alpha)$. Of course, in this case the homology is
  very sensitive to the deformations of $I$ and $\alpha$.
\end{Remark}

Finally, for $\fF\subset \tpi_1(P)$ set
$$
\HC_*^{I}(\alpha;\fF)=\bigoplus_{\fh\in\fF}\HC_*^{I}(\alpha;\fh).
$$
As is pointed out above, we will usually have $\fF=\{\ff^k\mid
k\in\N\}$ where all classes $\ff^k$ are distinct and none of these
classes is trivial.

In several instances we will also need to work with linearized contact
homology. Below we only briefly specify our conventions. For a
detailed discussion of the subject, we refer the reader to, e.g.,
\cite{Bo,SFT} and, in particular, to \cite[Section~3.1]{BO}.

In this case we start with a strong symplectic filling $W$ of
$(P,\xi=\ker\alpha)$, i.e., a compact symplectic manifold $(W,\omega)$
such that $\p W= P$ and $\omega|_P=d\alpha$ for some $\alpha$ and that
a natural orientation compatibility condition is satisfied.  The form
$\alpha$ need not be index--admissible, but the linearized contact
homology is defined only when a filling exists and the homology
depends on the filling. (However, a filling for $\alpha$ can be
adjusted and turned into a filling for any other form supporting $\xi$
without changing the total linearized contact homology.)  When working
with linearized contact homology, we need to replace $\tpi_1(P)$ by
$\tpi_1(W)$ everywhere in the above discussion. We use the notation
$\HC_*^I(\alpha;W,\fc)$ for the filtered linearized contact homology,
where now $\fc\in\tpi_1(W)$. This is a vector space over $\Q$ which
depends on $I$ and $\alpha$ and $\fc$ (as in the cylindrical case) and
also on $W$. Furthermore, when $\fc\neq 1$, we assume for the sake of
simplicity that the filling is exact (i.e., $[\omega]=0$) and that
$c_1(TW)=0$ in $H^2(P;\Z)$. (This condition can be relaxed.)

If $\fc=1$, it suffices to require that $W$ is symplectically
aspherical, i.e., $[\omega]|_{\pi_2(W)}=0=c_1(TW)|_{\pi_2(W)}$. Note
that in this case the contact action given by \eqref{eq:action} is, in
general, different from the symplectic area bounded by an orbit in
$W$, unless the orbits is contractible in $P$ or the filling is
exact. In what follows, the action is always taken to be the contact
action as defined by \eqref{eq:action}.

\subsection{Local contact homology}
\label{sec:lch}
Next, let us review the construction of the local contact homology and
the relevant results. Here we follow \cite{GHHM,HM} and we refer the
reader to these two papers for proofs and details.

Consider an isolated closed orbit $x$, not necessarily simple, of the
Reeb flow of a contact form $\alpha$. The local contact homology
$\HC_*(x)$ of $x$ is the homology of the complex $\CC_*(x,\tal)$
generated by the (good) periodic orbits which $x$ splits into under a
non-degenerate perturbation $\tal$ of $\alpha$ with the differential
defined again by counting rigid holomorphic cylinders in the
symplectization of a tubular neighborhood of $x$. The resulting
complex depends on $\alpha$, but its homology is well-defined. (Note
that here it is essential that $x$ is isolated.) The absolute grading
of the local contact homology is defined once we fix a trivialization
of $\xi |_x$. In the setting of Section \ref{sec:cch}, we can use for
instance the trivialization arising from our global trivialization
convention. As in the global case, throughout the paper \emph{the
  local contact homology is graded by the Conley--Zehnder index}.

For instance, when $x$ is non-degenerate and good, $\HC_*(x)$ is $\Q$
concentrated in one degree, equal to $\MUCZ(x)$. When $x$ is
non-degenerate and bad, $\HC_*(x)=0$.

The local contact homology of periodic orbits of $\alpha$ are building
blocks of $\HC_*(\xi;\fF)$. Namely, there exists a spectral sequence
with $E^1=\bigoplus_x\HC_*(x)$, where the sum is over all (not
necessarily simple) closed Reeb orbits of $\alpha$ in $\fF$,
converging to $\HC_*(\xi;\fF)$. As a result, $\HC_m(\xi;\fF)=0$ if
there exists a form $\alpha$ such that $\HC_m(x)=0$ for all $x$ in
$\fF$.

To be more precise, assume that all closed Reeb orbits of $\alpha$ are
isolated and, as a consequence, the action spectrum $\CS(\alpha)$ is
discrete: $\CS(\alpha)=\{c_1,c_2,\ldots\}$, where
$c_1<c_2<\cdots$. Pick arbitrary positive real numbers $a_i$
separating the points of the spectrum:
$$
0<a_0<c_1<a_1<c_2<a_2<c_3<\cdots .
$$
It is easy to see that
$$
\HC_*^{(a_p,\,a_{p+1})}(\alpha;\fF)=\bigoplus_x\HC_*(x),
$$
where the sum is taken over all $x$ in $\fF$ with $\CA(x)=c_p$.
Hence, the increasing filtration $\CC_*^{(0,\,a_p)}(\alpha;\fF)$ of
$\CC_*(\alpha)$ gives rise to a spectral sequence with $E^1_{p,q}=
\bigoplus_x\HC_{p+q}(x)$ converging to $\HC_*(\xi;\fF)$.

In general, $\HC_*(x)$ is supported in the interval
$[\Delta(x)-n,\Delta(x)+n]$ or, in other words, $\HC_m(x)$ can be
non-zero only for $m$ in this interval. (This readily follows from
\eqref{eq:indexes}.) Thus we can write, using self-explanatory
notation,
\begin{equation}
\label{eq:supp}
\supp \HC_*(x)\subset [\Delta(x)-n,\,\Delta(x)+n],
\end{equation}
and, when $x$ is weakly non-degenerate, the inclusion is strict at
both end-points of the interval, i.e., $\Delta(x)\pm n$ are not in the
support.

For our purposes it is essential to understand how the local contact
homology behaves under iterations. Let now $x$ be a simple closed Reeb
orbit. A positive integer $k$, the order of iteration, is said to be
\emph{admissible} if the Floquet multiplier 1 occurs with the same the
multiplicity for $x$ and $x^k$, i.e., $k$ is not divisible by the
order of any root of unity among the Floquet multipliers of $x$. For
instance, $k$ is admissible when both $x$ and $x^k$ are non-degenerate
or, as the opposite extreme, any $k$ is admissible when $x$ is
\emph{totally degenerate} (i.e., all Floquet multipliers are equal to
1). Furthermore, every sufficiently large prime is admissible
regardless of the Floquet multipliers of $x$.  Note also that an
admissible iteration of an isolated periodic orbit is automatically
isolated; see \cite{CMPY,GG:gap}.

For an isolated simple periodic orbit $x$, there exists a sequence
$s_k\in\Z$ indexed by all admissible iterations of $x$ such that
\begin{equation}
\label{eq:lch-bound}
\dim\HC_*(x^k)\leq \dim\HC_{*-s_k}(x)
\textrm{ where }
\lim_{k\to\infty}\frac{s_k}{k}=\Delta(x).
\end{equation}
Moreover, $s_k=\Delta(x)k$ when $x$ is totally degenerate. We refer
the reader to \cite{HM} for a proof of this fact; see also
\cite{GHHM}. In particular, when all iterations of $x$ are isolated,
the sequence $\dim \HC_*(x^k)$ is bounded as a function of $k$. These
results can be thought of as generalizations of the classical
Gromoll--Meyer theorem (see \cite{GM}) to contact homology; see also
\cite{McL} for an analogue of the Gromoll--Meyer theorem for
symplectic homology.

The proof of \eqref{eq:lch-bound} is based on the relation between the
local contact homology of an isolated orbit, say $y$, and the local
Floer homology $\HF_*(\psi)$ of its Poincar\'e return map
$\psi$. Namely, for a simple orbit $y$ we just have
$\HC_*(y)\cong\HF_*(\psi)$ with our grading conventions. When the
orbit is iterated, i.e., $y=x^k$ and $\psi=\varphi^k$ where $x$ is
simple, $\varphi$ is the Poincar\'e return map of $x$ and $k$ is
admissible, the relation is more involved. However, even in this case,
we still have $\dim\HC_*(y)\leq \dim \HF_*(\psi)$. (See \cite{GHHM,HM}
for the proofs. For a simple orbit, the result can also be established
by repeating word-for-word the proof of \cite[Proposition
4.30]{EKP}. The example where $x$ and $y$ are both non-degenerate and
$y$ is bad shows that a strict inequality does occur.) Finally, a
version of \eqref{eq:lch-bound} holds for the local Floer homology
(see \cite{GG:gap} for a proof), i.e., to be more precise,
$$
\dim\HF_*(\varphi^k)=\dim\HF_{*-s_k}(\varphi)
\textrm{ with }
\lim_{k\to\infty}\frac{s_k}{k}=\Delta(x),
$$
where again $s_k=\Delta(x)k$ when $x$ is totally degenerate, and
\eqref{eq:lch-bound} follows.

\subsection{Symplectically degenerate maxima}
\label{sec:sdm} 
A closed Reeb orbit $x$ is said to be a \emph{symplectically
  degenerate maximum} or SDM if $\HC_{\Delta(x)+n}(x)\neq 0$, i.e.,
the local contact homology of $x$ is non-trivial at the right
end-point of the maximal support interval. Such orbits are necessarily
totally degenerate, and $\HC_*(x)$ is $\Q$ when $*=\Delta(x)+n$ and
zero otherwise. An iteration of an SDM orbit is again an SDM. We refer
the reader to \cite{GHHM,HM} for the proofs of these facts. Note also
that, although both the mean index and the grading of $\HC_*(x)$
depend on the trivialization of $\xi|_x$, the notion of an SDM is
independent of the trivialization.

The role of SDM Reeb orbits in our proof of Theorem \ref{thm:main} is
similar to that of SDM orbits for Hamiltonian diffeomorphisms in the
proof of the Hamiltonian Conley conjecture;
\cite{Gi:CC,GG:gaps,Hi}. We have the following result where, for
technical reasons, we need to use the linearized contact homology; see
Remark \ref{rmk:difficulty}.

\begin{Theorem}[\cite{GHHM}]
\label{thm:sdm}
Let $(W,\omega)$ be a strong symplectic filling of a closed contact
manifold $(P,\xi)$ and let $\fc\in \tpi_1(W)$. Assume furthermore that
at least one of the following two conditions is satisfied:
\begin{itemize}

\item $W$ is symplectically aspherical, i.e.,
  $[\omega]|_{\pi_2(W)}=0=c_1(TW)|_{\pi_2(W)}$, and $\fc=1$, or

\item $[\omega]=0$ in $\H^2(P;\R)$ and $c_1(TW)=0$ in $H^2(P;\Z)$.

\end{itemize}
Let $x$ be a simple isolated closed Reeb orbit of a contact form
$\alpha$ on $(P,\xi)$. Assume that $x$ is an SDM with mean index
$\Delta=\Delta(x)$ and action $c=\A(x)$ and that $x$ is in the class
$\fc$. Then for any $\eps>0$ there exists $k_\eps\in \N$ such that
$$
\HC_{k\Delta+n+1}^{(kc,\,kc+\eps)}(\alpha;W,\fc^k)\neq 0\textrm{ for
  all } k>k_\eps .
$$
\end{Theorem}

This theorem is proved in \cite{GHHM}, although it is stated slightly
differently in that paper. We emphasize that in Theorem \ref{thm:sdm}
and Corollary \ref{cor:sdm} below the manifold $P$ need not be a
pre-quantization circle bundle, and hence these results apply to a
broader class of manifolds than Theorem \ref{thm:main}.

\begin{Remark}
\label{rmk:interval}
Note also that in Theorem \ref{thm:sdm}, as in similar results in the
Hamiltonian setting (see, e.g., \cite[Proposition 4.7]{Gi:CC},
\cite[Theorem 1.7]{GG:gaps} and \cite[Theorem 1.5]{He}), one should,
strictly speaking, replace the action interval by $(k c+\delta,\,k
c+\eps)$ for some arbitrarily small $\delta\in (0,\,\eps)$ and require
$\eps$ to be outside a certain zero measure set to make sure that the
end points of the interval are not in the action spectrum.
\end{Remark}

As a consequence of Theorem \ref{thm:sdm}, we obtain

\begin{Corollary}[\cite{GHHM}]
\label{cor:sdm}
In the setting of Theorem \ref{thm:sdm}, the Reeb flow of $\alpha$ has
infinitely many simple periodic orbits.
\end{Corollary}

Note however that this result, in such a general setting, affords no
control on the free homotopy classes of the simple orbits or their
growth rate. In contrast with its Hamiltonian counterpart, the
corollary is not entirely obvious. For the sake of completeness and
because the argument is used in the proof of Theorem \ref{thm:main},
we include a detailed proof of the corollary; cf.\ \cite[Section
3.2]{GG:gaps}.

\begin{proof}
  By Theorem \ref{thm:sdm}, given $\eps>0$, for every sufficiently
  large $k$, there exists a closed Reeb orbit $y_k$ such that
  $kc<\CA(y_k)<ck+\eps$.

  Arguing by contradiction, assume that $\alpha$ has only finitely
  many simple closed Reeb orbits $z_1,\ldots,z_r$. Then, for every
  large $k$, we have $y_k=z_i^{m_k}$ for at least one orbit $z_i$. Set
  $a_i=\CA(z_i)$. We have $kc<m_ka_i<kc+\eps$. Once $\eps<a_i$ for all
  $i$, it follows that
\begin{equation}
\label{eq:pf-cor}
0<\|kc\|_{a_i}<\eps \textrm{ for all } k \textrm { and } i,
\end{equation}
where $\|t\|_a$ stands for the distance from $t\in\R$ to the nearest
point in $a\Z$.

We will show that this is impossible: when $\eps>0$ is sufficiently
small there is a sequence $k_j\to\infty$ such that either
$\|k_jc\|_{a_i}=0$ or $\|k_jc\|_{a_i}>\eps$ for every $k_j$ and $i$.

We consider two cases: $c\in a_i\Q$ and $c\not\in a_i\Q$. In the
former case, there exists $\delta_i>0$ such that for every $k$ either
$\|kc\|_{a_i}=0$ or $\|kc\|_{a_i}>\delta_i$. In other words, when
$c\in a_i\Q$ and $\eps<\delta_i$, \eqref{eq:pf-cor} fails for all $k$.

In the latter case, the sequence $ck$ is equidistributed in the circle
$\R/a_i\Z$. Thus, for every $\delta<a_i/2$, we have
$\|kc\|_{a_i}>\delta$ with probability $1-2\delta/a_i$. It follows
that, when $\eps>0$ is sufficiently small, for all $i$ such that
$c\not\in a_i\Q$ the condition $\|kc\|_{a_i}>\eps$ is satisfied with
positive probability, i.e., for a positive density sequence
$k_j$. \end{proof}

\begin{Remark}[Symplectically Degenerate Minima] A sister notion of an
  SDM is that of a symplectically degenerate minimum (SDMin) obtained
  by replacing the right end point of the maximal support interval by
  the left end point. In other words, an isolated Reeb orbit $x$ is
  said to be an SDMin if $\HC_{\Delta(x)-n}(x)\neq 0$.  This notion is
  also of interest in Hamiltonian and contact dynamics; see
  \cite[Remark 1.3]{GHHM} and \cite{He:cot}. Symplectically degenerate
  maxima and minima have very similar properties, and variants of
  Theorem \ref{thm:sdm} and Corollary \ref{cor:sdm} hold when $x$ is
  an SDMin with now
  $\HC_{k\Delta-n-1}^{(kc-\eps,\,kc)}(\alpha;W,\fc^k)\neq 0$.
\end{Remark}

\begin{Remark}
\label{rmk:difficulty}
We expect an analogue of Theorem \ref{thm:sdm} to hold in the context
of cylindrical contact homology. However, the proof from \cite{GHHM}
does not readily translate to this setting without extra assumptions
on $\alpha$ along the lines of strong index positivity/negativity in
addition to $\alpha$ being index--admissible.  The difficulty is that
it is not clear how to make the forms $\alpha_\pm$, used in the proof
to ``estimate'' the contact homology of $\alpha$, index--admissible
without an extra condition of this type. This analogue of Theorem
\ref{thm:sdm} is of interest because, for instance, it would allow one
to eliminate the weak non-degeneracy requirement in the second part of
Theorem \ref{thm:main} and in the growth lower bounds from Remark
\ref{rmk:growth}.
\end{Remark}

\section{Proof of the contact Conley conjecture}
\label{sec:proof}

In this section we prove Theorems \ref{thm:main} and
\ref{thm:magnetic}, starting with some elementary preliminary
observations.

\subsection{The free homotopy class of the fiber}
\label{sec:fiber}
Let, as in Section \ref{sec:main}, $\pi\colon P\to M$ be a principle
$S^1$-bundle with the first Chern class $-[\omega]$, and let $\ff$ be
the free homotopy class of the fiber of $\pi$. In this section we show
that requirements (i-a) and (i-b) guarantee that the classes $\ff^k$,
$k\in\N$, satisfy the conditions mentioned in Section
\ref{sec:discussion} and needed for the proof of Theorem
\ref{thm:main}. For instance, we will show that $\ff$ is primitive and
all classes $\ff^k$ are distinct. Our first result is

\begin{Lemma}
\label{lemma:fiber}
Assume that condition (i-a) holds: $\omega|_{\pi_2(M)}=0$. Then
$\ff^k=\ff^l$ in $\tpi_1(P)$ only when $k=l$ and, in particular,
$\ff^k\neq 1$ for $k\neq 0$.
\end{Lemma}

Note that the lemma would be absolutely obvious if the class of the
fiber were non-zero in $\H_1(P;\Z)/\Tors$. However, clearly the image
of $\ff$ in the homology is a torsion class when $[\omega]\neq 0$, and
a proof is due.

\begin{proof} Consider the homotopy long exact sequence of $\pi\colon
  P\to M$. We claim that, since $\omega|_{\pi_2(M)}=0$, the connecting
  map $\p\colon \pi_2(M)\to \pi_1(S^1)$ is trivial. Indeed, the image
  $\p(s)$ of a class $s\in\pi_2(M)$ is equal to
  $\left<\omega,s\right>\cdot f$, where $f\in\pi_1(S^1)$ is the class
  of the fiber oriented by $R_0$. By the assumption,
  $\omega|_{\pi_2(M)}=0$, and we have $\p=0$.

  Thus the $\pi_1$-part of the long exact sequence turns into the
  short exact sequence
\begin{equation}
\label{eq:seq}
1\to\pi_1(S^1)\to\pi_1(P)\to\pi_1(M)\to 1 ,
\end{equation}
and hence $\pi_1(S^1)=\Z$ is a normal subgroup of $\pi_1(P)$. In other
words, for any $g\in \pi_1(P)$, the conjugation by $g$ is an
automorphism of $\pi_1(S^1)=\Z$. Then $gfg^{-1}=f^{\pm 1}$ because
$f^{\pm 1}$ are the only generators. (We are using here multiplicative
notation, for, in general, $\pi_1(P)$ is not commutative.) Moreover,
in fact, $gfg^{-1}=f$ since $\pi\colon P\to M$ is a principle
$S^1$-bundle and hence orientable. Now it follows that $f^k$ is
conjugate to $f^l$ only when $k=l$. In particular, $f^k$ is (conjugate
to) $1$ only when $k=0$.
\end{proof}

It also follows from the exact sequence \eqref{eq:seq} that the only
elements in $\tpi_1(P)$ which project to $1\in\tpi_1(M)$ are $\ff^k$,
$k\in\Z$, i.e., the elements of $\pi_1(S)$.  Next, we have

\begin{Lemma}
\label{lemma:fiber2}
Assume that condition (i-b) is met: $\pi_1(M)$ is torsion free. Then
for every $k\in\N$ the only solutions $\fh\in\tpi_1(P)$ and $l\geq 0$
of the equation $\fh^l=\ff^k$ are $\fh=\ff^r$, for some $r\in\N$, and
$l=k/r$. (In particular, $\ff$ is primitive.)
\end{Lemma}

\begin{proof}
  Clearly, it is sufficient to show that the equation $h^l=f^k$ in
  $\pi_1(P)$ has no other solutions than $h=f^r$ and
  $l=k/r$. Projecting to $M$ and denoting the image of $h$ in
  $\pi_1(M)$ by $\bar{h}$, we arrive at $\bar{h}^l=1$. By (i-b),
  $\bar{h}=1$. Hence, $h\in\pi_1(S^1)=\Z$ and the result follows.
\end{proof}

\subsection{Cylindrical contact homology of a pre-quantization circle
  bundle}
\label{sec:ch-P}
Although our proof of Theorem \ref{thm:main} would go through with any
choice of trivializations described in Section \ref{sec:cch}, it is
more convenient to specialize this choice further. Namely, we take the
fiber of $\pi$ over a point $p\in M$ as a reference loop in the class
$\ff$ and the pull-back of a frame in $T_pM$ as the reference
trivialization. Note that then the $k$-th iteration of the fiber is
the reference loop for $\ff^k$, and the reference trivialization for
$\ff^k$ is still the pull-back of a frame in $T_pM$.

With this choice of trivializations, for all $k\in \N$ we have
\begin{equation}
\label{eq:ch}
\HC_*(\xi;\ff^k)=\H_{*+n}(M;\Q).
\end{equation}
In particular, $\HC_n(\xi;\ff^k)=\Q$ for all $k\in\N$.  The
isomorphism \eqref{eq:ch} immediately follows from the Morse-Bott
description of contact homology; see \cite{Bo:thesis,Bo}. Indeed, up
to a shift of degree equal to the mean index (for the Reeb flow of
$\alpha_0$) of the $k$th iteration of the fiber, the contact homology
$\HC_*(\xi;\ff^k)$ is the homology of the base $M$ with the Floer
homology grading, i.e., the grading shifted down by $n$. (We recall
again that throughout the paper the contact homology is graded by the
Conley--Zehnder index.)  Finally, for our choice of trivializations,
the mean index of the (iterated) fiber is zero, and we arrive
at~\eqref{eq:ch}.

\subsection{Proof of Theorem \ref{thm:main}: the contact Conley
  conjecture}
\label{sec:actual_proof}
We may assume that $\alpha$ has finitely many closed Reeb orbits in
the class $\ff$. (Otherwise there is nothing to prove.) Denote these
orbits by $x_1,\ldots,x_r$. Note that all of these orbits are
simple. Indeed, by Lemma \ref{lemma:fiber2}, the class $\ff$ is
primitive, but if one of the orbits $x_i$ were iterated so would be
the class $\ff$. As in the Hamiltonian setting, the proof splits into
two cases.

Assume first that one of the orbits $x_i$, say $x=x_1$, is an
SDM. This is the ``degenerate case'' of the theorem. Consider the unit
disk bundle $E=P\times D^2/S^1$ over $M$ equipped with symplectic form
$\omega_E=[\pi^*\omega+d(\rho\alpha_0)]/2$, where $\rho=|z|^2$, $z\in
D^2$, and $\alpha_0$ is a connection 1-form. (See Section
\ref{sec:main}.)  The symplectic manifold $(E,\omega_E)$ is a strong
symplectic filling of $(P,\alpha_0)$. Rescaling the form and applying
a contact isotopy, we obtain a strong filling, which we simply denote
by $W$, of $(P,\alpha)$. Clearly, $M\hookrightarrow W$ is a homotopy
equivalence, and the inclusion $P\hookrightarrow W$ is homotopic, in
the obvious sense, to the projection $P\to M$. The fiber of $P$ is
contractible in $W$, i.e., the image of the homotopy class $\ff\in
\tpi_1(W)$ of the fiber is $\fc=1\in \tpi_1(W)$. In general, the
filling $W$ is neither exact nor does it have zero first Chern
class. However, $\pi_2(W)=0$ since $M$ is aspherical, and, in
particular, $W$ is symplectically aspherical. Hence, Theorem
\ref{thm:sdm} and Corollary \ref{cor:sdm} apply with $\fc=1$, and the
Reeb flow of $\alpha$ has infinitely many periodic orbits. However,
Corollary \ref{cor:sdm} provides no information on the homotopy
classes of these orbits and an extra argument is needed to show that
the orbits have contractible projections to $M$ or, equivalently, are
contractible in $W$.

Arguing by contradiction, assume that the flow has only finitely many
simple orbits $\{z_i\}$ contractible in $W$. The orbits $x_i$ are
among these orbits, but there can be other orbits with contractible
projections to $M$. By Theorem \ref{thm:sdm}, given $\eps>0$, for
every sufficiently large $k$, the Reeb flow has a closed Reeb orbit
$y_k$ contractible in $W$ with action in the range
$(kc,\,kc+\eps)$. The orbit $y_k$ need not be simple.

As is pointed out in Section \ref{sec:fiber}, an element $\fy\in
\tpi_1(P)$ (e.g., $[y_k]$ or $[z_i]$) is trivial in
$\tpi_1(W)=\tpi_1(M)$ if and only if $\fy=\ff^l$ for some
$l\in\Z$. Thus $y_k$ can only be an iteration of one of the orbits
$z_i$, i.e., $y_k=z_i^{m_k}$. Now exactly the same argument as the
proof of Corollary \ref{cor:sdm} shows that this is impossible.

The second case is when none of the orbits $x_i$ is an SDM. This is
the so-called ``non-degenerate case'' of the theorem since, for
instance, none of these orbits is an SDM when all $x_i$ are weakly
non-degenerate. In spirit, the proof of this case goes back to
\cite{SZ}. Let $k$ be a sufficiently large prime. To prove the
theorem, it suffices to show that the class $\ff^k$ contains a simple
periodic orbit. Then, by Lemma \ref{lemma:fiber}, all iterates
$\ff^k$, $k\in\N$, are distinct and hence so are the orbits.

Set $\Delta_i=\Delta(x_i)$.  We require $k$ to be admissible for all
orbits with $\Delta_i=0$ and large enough to ensure that
$k|\Delta_i|>2n$ for all orbits with $\Delta_i\neq 0$.

To show that there exists a closed Reeb orbit in the class $\ff^k$, we
again argue by contradiction. Indeed, assume the contrary. Then, by
Lemma \ref{lemma:fiber2}, every orbit in $\ff^k$ is necessarily of the
form $x_i^k$. We claim that $\HC_n(\alpha;\ff^k)=0$, which contradicts
\eqref{eq:ch}. It suffices to show that $\HC_n(x_i^k)=0$ for all
orbits $x_i$. To prove this, note that
$$
\supp \HC_*(x_i^k)\subset [k\Delta_i-n,\,k\Delta_i+n],
$$
by \eqref{eq:hom} and \eqref{eq:supp}, and hence $\HC_n(x_i^k)=0$ when
$\Delta_i\neq 0$. When $\Delta_i=0$, there are further two cases to
consider. The first one is when $x_i$ is weakly non-degenerate. Then
so is $x_i^k$, since $k$ is admissible, and hence
$\HC_n(x_i^k)=0$. (For the end-points of the interval in
\eqref{eq:supp} are not in the support.)  The second case is when
$x_i$ is totally degenerate. Then $\dim\HC_*(x_i^k)\leq
\dim\HC_*(x_i)$ by \eqref{eq:lch-bound} together with the ``moreover''
part and, in particular, $\HC_n(x_i^k)=\HC_n(x_i)=0$ because $x_i$ is
not an SDM. \hfill $\qed$

\begin{Remark}
\label{rmk:sdm_enough}
As is clear from the proof above, one can relax the assumption that
the orbits are weakly non-degenerate in the second part of the theorem
and replace it by the requirement that none of the orbits is an SDM.
\end{Remark}

\subsection{Proof of Theorem \ref{thm:magnetic}}
\label{sec:pf-magn} The level $P_\eps$ is a circle bundle over $M$,
isomorphic to the unit circle bundle $P\subset T^*M$. As is well
known, when $g\neq 1$, the form $\omega$ has contact type on $P_\eps$
for all small $\eps>0$. Indeed, the first Chern class of the
$S^1$-bundle $\pi\colon P\to M$ is equal to the Euler class $e(M)$.
Since $g\neq 1$, there exists $\kappa\in\R$ such that $[\sigma]=\kappa
e(M)$. Fix a connection form $\alpha_0$ on $P$ with curvature
$\sigma/\kappa$ as in Section \ref{sec:main} and set
$\lambda_0=\kappa\alpha_0$. Clearly, $d\lambda_0=\pi^*\sigma$ and
$\lambda_0$ is a contact form with Reeb vector field $\kappa^{-1}
R_0$. Next denote by $\lambda_1$ the restriction of the standard
Liouville form $pdq$ to $P$. Let us identify $P_\eps$ with $P$ via the
fiberwise dilation by $\sqrt{2\eps}$. Then the pull back of
$\omega|_{P_\eps}$ to $P$ is
$d(\lambda_0+\sqrt{2\eps}\lambda_1)$. When $\eps>0$ is small, this is
a contact form.

Denote by $\alpha$ the resulting contact primitive of
$\omega|_{P_\eps}$. In other words, $\alpha$ is the push-forward of
$\lambda_0+\sqrt{2\eps}\lambda_1$ to $P_\eps$. The underlying contact
structure $\ker\alpha$ is isotopic, for all small $\eps>0$, to the
pre-quantization contact structure $\ker\lambda_0=\ker\alpha_0$.  It
readily follows now that when $g\geq 2$ conditions (i) and (ii) of
Theorem \ref{thm:main} are satisfied.

Furthermore, the Reeb flow on $P_\eps$ is index--admissible. Indeed,
first note that every closed, homologous to zero Reeb orbit $x$ on
$P_\eps$ has mean index $\Delta(x)<2\chi(M)+o(1)$ as $\eps\to 0$; see
\cite{Be}. Fix now a sufficiently large $T$, depending on the geometry
of the magnetic field and the metric. Then for any small
non-degenerate perturbation $\tal$ of $\alpha$, every closed Reeb
orbit $\tx$ of $\tal$ of period less than $T$ is close to an orbit of
$\alpha$, and hence also has mean index $\Delta(\tx)<2\chi(M)+o(1)$
when $\tx$ is contractible in $P$. Long orbits $\tx$ of $\tal$, i.e.,
the orbits with period greater than or equal to $T$, do not
necessarily arise from the orbits of $\alpha$. However, such orbits
automatically have large negative mean index $-O(T)<2\chi(M)$ (as
$T\to\infty$); see \cite{GG:wm}. In either case, the Conley--Zehnder
index of $\tx$ is no greater than $2\chi(M)+1+o(1)<0$ by
\eqref{eq:indexes}.

To finish the proof of the first part of the theorem, the existence of
infinitely many periodic orbits with contractible projections to $M$,
it remains to apply Theorem~\ref{thm:main}.

To show that there is indeed a simple orbit in the class $\ff^k$ for
every large prime $k$, we can focus on the case where one of the
simple orbits $x_1,\ldots,x_r$ in the class $\ff$ is an SDM. For the
``non-degenerate case'' is also covered by (the proof of) Theorem
\ref{thm:main}; see Remark \ref{rmk:sdm_enough}. We do this by
applying Theorem \ref{thm:sdm} to a conveniently chosen filling of
$P_\eps$ with an extra component added to it. Namely, let us fix a
metric with constant negative curvature on $M$ and let $P'\subset
T^*M$ be a high energy level with respect to this metric.  Clearly,
$P'$ is a contact type hypersurface in $T^*M$, and the Reeb flow on
$P'$ is topologically conjugate to the hyperbolic geodesic flow on
$M$; see, e.g., \cite {Gi:newton}. (Moreover, one can find such a
hyperbolic metric with area form proportional, with constant factor,
to the magnetic field $\sigma$. Then the twisted geodesic flow on $P'$
is smoothly conjugate to the geodesic flow; cf.\ \cite{Ar61}.)  As a
consequence, all closed orbits of the flow on $P'$ have
non-contractible projections to $M$.

Consider now the subset $W$ of $T^*M$ bounded by $P'$ and
$P_\eps$. This is a strong filling of $P'\sqcup P_\eps$; see
\cite{McD} and also \cite{Ge}. The filling $W$ is exact and
$c_1(TW)=0$. Hence, Theorem \ref{thm:sdm} applies with $\fc=\ff$, the
homotopy class of the fiber. The inclusion $P_\eps\hook W$ is a
homotopy equivalence, and the Reeb flow on $P'$ has no closed orbits
in any class $\ff^k$. Let $x$ be an SDM in the class $\ff$. Note that
now, as in Section \ref{sec:ch-P}, we can assume without loss of
generality that $\Delta(x)=0$.  By Theorem \ref{thm:sdm} applied to
$W$, for every sufficiently large $k$, there exists a periodic orbit
$y$ of the flow on $P_\eps$ such that $[y]\in\ff^k$ and
$0<\Delta(y)<2n+1$. By Lemma \ref{lemma:fiber}, when $k$ is prime,
either $y$ is simple or it is an iteration of one of the orbits $x_i$
with $\Delta(x_i)>0$. The latter is clearly impossible when $k$ is so
large that $k\Delta(x_i)>2n+1$ for all $x_i$ with $\Delta(x_i)>0$.
This completes the proof of the theorem.  \hfill $\qed$


\begin{thebibliography}{BEHWZ}


\bibitem[AMP]{AMP} A. Abbondandolo, L. Macarini, G.P. Paternain, On
  the existence of three closed magnetic geodesics for subcritical
  energies, Preprint arXiv:1305.1871; to appear in \emph{Comm.\ Math.\
    Helv.}

\bibitem[AM$^2$P]{AMMP} A. Abbondandolo, L. Macarini, M. Mazzucchelli,
  G.P. Paternain, Infinitely many periodic orbits of exact magnetic
  flows on surfaces for almost every subcritical energy level,
  Preprint arXiv:1404.7641.

\bibitem[AS]{AS}
A. Abbondandolo, M. Schwarz, On the Floer homology of cotangent bundles, 
\emph{Comm.\ Pure Appl.\ Math.},  \textbf{59}  (2006),  254--316.

\bibitem[AM]{AM}
M. Abreu, L. Macarini, Contact homology of good toric contact manifolds, 
\emph{Compos.\ Math.}, \textbf{148} (2012), 304--334. 

\bibitem[Ar61]{Ar61} V.I. Arnold, Some remarks on flows of line
  elements and frames, \emph{Soviet Math.\ Dokl.}, \textbf{2} (1961),
  562--564.

\bibitem[Ar86]{Ar:fs} V.I. Arnold, First steps of symplectic topology,
  \emph{Russ.\ Math.\ Surveys}, \textbf{41} (1986), 1--21.

\bibitem[Ar88]{Ar88} V.I. Arnold, On some problems in symplectic
  topology, in \emph{Topology and Geometry -- Rochlin Seminar}, O.Ya.\
  Viro (ed.), Lect.\ Notes in Math., vol.\ 1346, Springer, 1988.

\bibitem[Be]{Be} G. Benedetti, The contact property for nowhere
  vanishing magnetic fields on the two-sphere, Preprint
  arXiv:1308.2128.

\bibitem[Bo02]{Bo:thesis} F. Bourgeois, Morse--Bott approach to
  contact homology, Ph.D. thesis, Stanford, 2002.

\bibitem[Bo09]{Bo} F. Bourgeois, A survey of contact homology,
  \emph{CRM Proceedings and Lecture Notes}, \textbf{49} (2009),
  45--60.

\bibitem[BO]{BO} F. Bourgeois, A. Oancea, An exact sequence for
  contact- and symplectic homology \emph{Invent.\ Math.}, \textbf{175}
  (2009), 611--680.

\bibitem[BvK]{BvK} F. Bourgeois, O. van Koert, Contact homology of
  left-handed stabilizations and plumbing of open books,
  \emph{Commun.\ Contemp.\ Math.}, \textbf{12} (2010), 223--263.

\bibitem[BH]{BH} B. Bramham, H. Hofer, First steps towards a
  symplectic dynamics, \emph{Surveys in Differential Geometry},
  \textbf{17} (2012),127--178.

\bibitem[CGG]{CGG} M. Chance, V.L. Ginzburg, B.Z. G\"urel,
  Action-index relations for perfect Hamiltonian diffeomorphisms,
  \emph{J.\ Symplectic Geom.}, \textbf{11} (2013), 449--474. 

\bibitem[CMPY]{CMPY} S.-N. Chow, J. Mallet-Paret, J.A. Yorke, A
  periodic orbit index which is a bifurcation
  invariant. \emph{Geometric Dynamics (Rio de Janeiro, 1981)},
  109--131, Lecture Notes in Math., 1007, Springer, Berlin, 1983.

\bibitem[CMP]{CMP} G. Contreras, L. Macarini, G.P. Paternain, Periodic
  orbits for exact magnetic flows on surfaces, \emph{Int.\ Math.\
    Res.\ Not.\ IMRN}, 2004, no.\ 8, 361--387.

\bibitem[El]{El} Y. Eliashberg, Invariants in contact topology, in
  \emph{Proceedings of the International Congress of Mathematicians,
    Vol.\ II (Berlin, 1998)}, Doc.\ Math.\ 1998, Extra Vol.\ II,
  327--338.

\bibitem[EGH]{SFT} Y. Eliashberg, A. Givental, H. Hofer, Introduction
  to symplectic field theory, \emph{Geom.\ Funct.\ Anal.}, 2000,
  Special Volume, Part II, 560--673.

\bibitem[EKP]{EKP} Y. Eliashberg, S.S. Kim, L. Polterovich, Geometry
  of contact transformations and domains: orderability vs. squeezing,
  \emph{Geom.\ Topol.}, \textbf{10} (2006), 1635--1747.


\bibitem[Es]{Es} J. Espina, On the mean Euler characteristic of
  contact structures, Preprint arXiv:1011.4364; to appear in
  \emph{Internat.\ J. Math.}

\bibitem[Fr92]{Fr1} J. Franks, Geodesics on $S^2$ and periodic points
  of annulus homeomorphisms, \emph{Invent.\ Math.}, \textbf{108}
  (1992), 403--418.

\bibitem[Fr96]{Fr2} J. Franks, Area preserving homeomorphisms of open
  surfaces of genus zero, \emph{New York Jour.\ of Math.}, \textbf{2}
  (1996) 1--19.

\bibitem[FH]{FH} J. Franks, M. Handel, Periodic points of Hamiltonian
  surface diffeomorphisms, \emph{Geom.\ Topol.}, \textbf{7} (2003),
  713--756.

\bibitem[Ge]{Ge} H. Geiges, Examples of symplectic 4-manifolds with
  disconnected boundary of contact type, \emph{Bull.\ London
    Math. Soc.}, \textbf{27} (1995), 278--280.

\bibitem[Gi87]{Gi:FA} V.L. Ginzburg, New generalizations of
  Poincar\'e's geometric theorem, \emph{Funct.\ Anal.\ Appl.},
  \textbf{21} (1987), 100--106.


\bibitem[Gi96a]{Gi:newton} V.L. Ginzburg, On closed trajectories of a
  charge in a magnetic field. An application of symplectic geometry in
  \emph{Contact and symplectic geometry (Cambridge, 1994)}, 131--148,
  Publ.\ Newton Inst., 8, Cambridge Univ.\ Press, Cambridge, 1996.

\bibitem[Gi96b]{Gi:mathz} V.L. Ginzburg, On the existence and
  non-existence of closed trajectories for some Hamiltonian flows,
  \emph{Math.\ Z.}, \textbf{223} (1996), 397--409.

\bibitem[Gi10]{Gi:CC} V.L. Ginzburg, The Conley conjecture,
  \emph{Ann.\ of Math.}, \textbf{172} (2010), 1127--1180.

\bibitem[GGo]{GGo} V.L. Ginzburg, Y. G\"oren, Iterated index and the
  mean Euler characteristic, Preprint arXiv:1311.0547.

\bibitem[GG04]{GG:capacity} V.L. Ginzburg, B.Z. G\"urel, Relative
  Hofer-Zehnder capacity and periodic orbits in twisted cotangent
  bundles, \emph{Duke Math.\ J.}, \textbf{123} (2004), 1--47.

\bibitem[GG07]{GG:wm} V.L. Ginzburg, B.Z. G\"urel, Periodic orbits of
  twisted geodesic flows and the Weinstein--Moser theorem,
  \emph{Comment.\ Math.\ Helv.}, \textbf{84} (2009), 865--907.


\bibitem[GG09]{GG:gaps} V.L. Ginzburg, B.Z. G\"urel, Action and index
  spectra and periodic orbits in Hamiltonian dynamics, \emph{Geom.\
    Topol.}, \textbf{13} (2009), 2745--2805.

\bibitem[GG10]{GG:gap} V.L. Ginzburg, B.Z. G\"urel, Local Floer
  homology and the action gap, \emph{J. Sympl.\ Geom.}, \textbf{8}
  (2010), 323--357.

\bibitem[GG12]{GG:nm} V.L. Ginzburg, B.Z. G\"urel, Conley conjecture
  for negative monotone symplectic manifolds, \emph{Int.\ Math.\ Res.\
    Not.\ IMRN}, 2012, no.\ 8, 1748--1767.

\bibitem[GH$^2$M]{GHHM} V.L. Ginzburg, D. Hein, U.L. Hryniewicz,
  L. Macarini, Closed Reeb orbits on the sphere and symplectically
  degenerate maxima, \emph{Acta Math.\ Vietnam.}, \textbf{38} (2013),
  55--78

\bibitem[GM]{GM} D. Gromoll, W. Meyer, Periodic geodesics on compact
  Riemannian manifolds, \emph{J. Differential Geom.}, \textbf{3}
  (1969), 493--510.

\bibitem[GGK]{GGK} V. Guillemin, V. Ginzburg, Y. Karshon,
  \emph{Cobordisms and Hamiltonian Group Actions}, Mathematical
  Surveys and Monographs, 98; American Mathematical Society,
  Providence, RI, 2002.

\bibitem[G\"u12]{Gu:hyp} B.Z. G\"urel, Periodic orbits of Hamiltonian
  systems linear and hyperbolic at infinity, Preprint arXiv:1209.3529;
  to appear in \emph{Pacific J. Math.}

\bibitem[G\"u13]{Gu:nc} B.Z. G\"urel, On non-contractible periodic
  orbits of Hamiltonian diffeomorphisms, \emph{Bull.\ Lond.\ Math.\
    Soc.}, 2013, doi: 10.1112/blms/bdt051.

\bibitem[G\"u14]{Gu:pr} B.Z. G\"urel, Perfect Reeb flows and
  action-index relations, Preprint arXiv:1401.2665.

\bibitem[He11]{He:cot} D. Hein, The Conley conjecture for the
  cotangent bundle, \emph{Arch.\ Math.}, \textbf{96} (2011), 85--100.

\bibitem[He12]{He} D. Hein, The Conley conjecture for irrational
  symplectic manifolds, \emph{J. Sympl.\ Geom.}, \textbf{10} (2012),
  183--202.

\bibitem[Hi93]{Hi:geod} N. Hingston, On the growth of the number of
  closed geodesics on the two-sphere, \emph{Int.\ Math.\ Res.\ Not.\
    IMRN}, 1993, no.\ 9, 253--262.

\bibitem[Hi09]{Hi} N. Hingston, Subharmonic solutions of Hamiltonian
  equations on tori, \emph{Ann.\ of Math.}, \textbf{170} (2009),
  525--560.

\bibitem[HWZ98]{HWZ:convex} H. Hofer, K. Wysocki, E. Zehnder, The
  dynamics on three-dimensional strictly convex energy surfaces,
  \emph{Ann.\ of Math.}, \textbf{148} (1998), 197--289.

\bibitem[HWZ10]{HWZ:SC} H. Hofer, K. Wysocki, E. Zehnder,
  SC-smoothness, retractions and new models for smooth spaces,
  \emph{Discrete Contin.\ Dyn.\ Syst.}, \textbf{28} (2010), 665--788.

\bibitem[HWZ11]{HWZ:poly} H. Hofer, K. Wysocki, E. Zehnder,
  Applications of polyfold theory I: The Polyfolds of Gromov--Witten
  Theory, Preprint arXiv:1107.2097.

\bibitem[HM]{HM} U. Hryniewicz, L. Macarini, Local contact homology
  and applications, Preprint arXiv:1202.3122.

\bibitem[Ka]{Ka} A.B. Katok, Ergodic perturbations of degenerate
  integrable Hamiltonian systems, \emph{Izv.\ Akad.\ Nauk SSSR Ser.\
    Mat.}, \textbf{37} (1973), 539--576.

\bibitem[Ke99]{Ke:m} E. Kerman, Periodic orbits of Hamiltonian flows
  near symplectic critical submanifolds, \emph{Int.\ Math.\ Res.\
    Not.\ IMRN}, 1999, no.\ 17, 953--969.

\bibitem[Ke12]{Ke} E. Kerman, On primes and period growth for
  Hamiltonian diffeomorphisms, \emph{J. Mod.\ Dyn.}, \textbf{6}
  (2012), 41--58.

\bibitem[LeC]{LeC} P. Le Calvez, Periodic orbits of Hamiltonian
  homeomorphisms of surfaces, \emph{Duke Math.\ J.}, \textbf{133}
  (2006), 125--184.

\bibitem[Lo]{Lo}
Y. Long,  \emph{Index Theory for Symplectic Paths with Applications}, 
Progress in Mathematics, 207, Birkh\"auser Verlag, Basel, 2002. 

\bibitem[L\"u]{Lu} W. L\"uck, Aspherical manifolds, \emph{Bulletin of
    the Manifold Atlas} (2012), 1--17, at
  \url{http://www.map.mpim-bonn.mpg.de/Aspherical_manifolds}.

\bibitem[Maz]{Ma}
M. Mazzucchelli,
 Symplectically degenerate maxima via generating functions, 
\emph{Math.\ Z.}, \textbf{275} (2013), 715--739.

\bibitem[McD]{McD} D. McDuff, Symplectic manifolds with contact type
  boundaries, \emph{Invent.\ Math.}, \textbf{103} (1991), 651--671.

\bibitem[McL]{McL} M. McLean, Local Floer homology and infinitely many
  simple Reeb orbits, \emph{Algebr.\ Geom.\ Topol.}, \textbf{12}
  (2012), 1901--1923.

\bibitem[SW]{SW} D. Salamon, J. Weber, Floer homology and the heat
  flow, \emph{Geom.\ Funct.\ Anal.}, \textbf{16} (2006), 1050--1138.

\bibitem[SZ]{SZ} D. Salamon, E. Zehnder, Morse theory for periodic
  solutions of Hamiltonian systems and the Maslov index, \emph{Comm.\
    Pure Appl.\ Math.}, \textbf{45} (1992), 1303--1360.

\bibitem[Scl]{Sc} F. Schlenk, Applications of Hofer’s geometry to
  Hamiltonian dynamics, \emph{Comment.\ Math.\ Helv.}, \textbf{81}
  (2006), 105--121.

\bibitem[Scn11]{Scn1} M. Schneider, Closed magnetic geodesics on
  $S^2$, \emph{J. Differential Geom.}, \textbf{87} (2011), 343--388.

\bibitem[Scn12]{Scn2} M. Schneider, Closed magnetic geodesics on
  closed hyperbolic Riemann surfaces, \emph{Proc.\ Lond.\ Math.\ Soc.\
    (3)}, \textbf{105} (2012), 424--446.

\bibitem[Ta]{Ta} I.A. Taimanov, Closed extremals on two-dimensional
  manifolds, \emph{Russian Math.\ Surveys} \textbf{47} (1992), no.\ 2,
  163--211.

\bibitem[Us]{Us} M. Usher, Floer homology in disc bundles and
  symplectically twisted geodesic flows, \emph{J. Mod. Dyn.},
  \textbf{3} (2009), 61--101.

\bibitem[VPS]{VPS} M. Vigu\'e-Poirrier, D. Sullivan, The homology
  theory of the closed geodesic problem, \emph{J.  Differential
    Geometry}, \textbf{11} (1976), 633--644.

\bibitem[Vi92]{Vi} C. Viterbo, Symplectic topology as the geometry of
  generating functions, \emph{Math.\ Ann.}, \textbf{292} (1992),
  685--710.

\bibitem[Vi99]{Vi:f} C. Viterbo, Functors and computations in Floer
  cohomology, I, \emph{Geom.\ Funct.\ Anal.}, \textbf{9} (1999),
  985--1033.

\bibitem[Yau]{Yau} M.-L. Yau, Vanishing of the contact homology of
  overtwisted contact 3--manifolds, \emph{Bull.\ Inst.\ Math.\ Acad.\
    Sin.\ (N.S.)}, \textbf{1} (2006), 211--229.

\bibitem[Zi]{Zi} W. Ziller, Geometry of the Katok examples,
  \emph{Ergodic Theory Dynam.\ Systems}, \textbf{3} (1983), 135--157.

\end{thebibliography}
\end{document}